\documentclass[a4paper]{article}
\usepackage{graphicx} 
\usepackage{amsmath, amsthm, amsfonts, amssymb, amscd, bm, enumerate}
\usepackage{verbatim}

\usepackage{url}

\newtheorem{theorem}{Theorem}
\newtheorem{corollary}{Corollary}
\newtheorem{lemma}{Lemma}
\newtheorem{definition}{Definition}

\theoremstyle{remark}
\newtheorem{remark}{Remark}
\newtheorem{example}{Example}
\def\F{\mathcal{F}}
\def\E{\mathcal{E}}

\def\bR{\mathbb{R}}

\def\bu{\mathbf{u}}
\def\bf{\mathbf{f}}
\def\bP{\mathbb{P}}

\def\bone{\mathbf{1}}

\def\Tr{\mathrm{Tr}}
\def\diag{\mathrm{diag}}

\begin{document}
\title{On Markovian solutions to Markov Chain BSDEs}
\author{Samuel N. Cohen\\Lukasz Szpruch\\ University of Oxford}
\date{\today}

\maketitle
\begin{quotation}
\emph{Dedicated to Charles Pearce on the occasion of his 70th Birthday}
\end{quotation}

\begin{abstract}
We study (backward) stochastic differential equations with noise coming from a finite state Markov chain. We show that, for the solutions of these equations to be `Markovian', in the sense that they are deterministic functions of the state of the underlying chain, the integrand must be of a specific form. This allows us to connect these equations to coupled systems of ODEs, and hence to give fast numerical methods for the evaluation of Markov-Chain BSDEs.


\end{abstract}

\section{Introduction}
Over the past 20 years, the role of stochastic methods in control has been increasing. In particular, the theory of Backward Stochastic Differential Equations, initiated by Pardoux and Peng \cite{Pardoux1990}, has shown itself to be a useful tool for the analysis of a variety of stochastic control problems (see, for example, El Karoui, Peng and Quenez \cite{El1997} for a review of applications in finance, or Yong and Zhou \cite{Yong1999} for a more general control perspective). Recent work \cite{Cohen2008, Cohen2008b} has considered these equations where noise is generated by a continuous-time finite-state Markov Chain, rather than by a Brownian motion.

Applications of BSDEs frequently depend on the ability to compute solutions to these equations numerically. While part of the power of the theory of BSDEs is its ability to deal with non-Markovian control problems, the numerical methods that have been developed are typically still restricted to the Markovian case (see, for example, \cite{Bouchard2004, Bender2007}). In this paper, we ask the question
\begin{quotation}
 When does a (B)SDE with underlying noise from a Markov Chain admit a `Markovian' solution, that is, one which can be written as a deterministic function of the current state of the chain?
\end{quotation}
As we shall see, such a property implies strong restrictions on the parameters of the (B)SDE. However, these restrictions form a type of nonlinear Feynman-Kac result, connecting solutions of these SDEs to solutions of coupled systems of ODEs. This connection yields simple methods of obtaining numerical solutions to a wide class of BSDEs in this context.

\section{Markov Chains and SDEs}
\subsection{Martingales and Markov Chains}
Consider a continuous-time finite-state Markov chain $X$ on a probability space $(\Omega, \bP)$. (The case where $X$ is a countable state process can also be treated in this manner, we exclude it only for technical simplicity.) Without loss of generality, we shall represent $X$ as taking values from the standard basis vectors $e_i$ of $\bR^N$, where $N$ is the number of states. An element $\omega\in\Omega$ can be thought of as describing a path of the chain $X$.

Let $\{\F_t\}$ be the completion of the filtration generated by $X$, that is,
\[\F_t = \sigma(\{X_s\}_{s\leq t}) \vee \{A\in\F:\bP(A)=0\}.\]
 As $X$ is a right-continuous pure jump process which does not jump at time $0$, this filtration is right-continuous. We assume that $X_0$ is deterministic, so $\F_0$ is the completion of the trivial $\sigma$-algebra.

Let $A$ denote the rate matrix\footnote{In our notation, $A$ is the matrix with entries $A_{ij}$, where $A_{ij}$ is the rate of jumping from state $j$ to state $i$. Depending on the convention used, this is either the rate matrix or its transpose.} of the chain $X$. As we do not assume time-homogeneity, $A$ is permitted to vary (deterministically) through time. We shall assume for simplicity that the rate of jumping from any state is bounded, that is, all components of $A$ are uniformly bounded in time. Note that $(A_t)_{ij}\geq 0$ for $i\neq j$ and $\sum_i A_{ij} = 0$ for all $j$ (the columns of $A$ all sum to $0$).

It will also be convenient to assume that $\bP(X_t = e_i)>0$ for any $t>0$ and any basis vector $e_i\in\bR^N$, that is, there is instant access from our starting state to any other state of the chain. None of our results depend on this assumption in any significant way, however without it, we shall be constantly forced to specify very peculiar null-sets, for states which cannot be accessed before time $t$. If we were to assume time-homogeneity (that is, $A$ is constant in $t$), this assumption would simply be that our chain is irreducible. However, this assumption does not mean that $A_{ij}>0$ for all $i\neq j$.

From a notational perspective, as $e_i$ denotes the $i$th standard basis vector in $\mathbb{R}^N$, the $i$th component of a vector $v$ is written $e_i^*v$, where $[\cdot]^*$ denotes vector transposition. For example, this implies that useful quantities can be written simply in terms of vector products. For example, we have $I_{X_t=e_i} = e_i^*X_t$.

\subsection{Markov-Chain SDEs}
We now relate our Markov chain to a $N$-dimensional martingale process, with which we can study SDEs. To do this, we write our chain in the following way
\[X_t = X_0 + \int_{]0,t]} A_u X_{u-} du + M_t\]
where $M$ is a locally-finite-variation pure-jump martingale in $\bR^N$. Our attention is then on the properties of stochastic integrals with respect to $M$.

We shall make some use of the following seminorm, which arises from the It\=o isometry.
\begin{definition}
Let $Z$ be a vector in $\bR^N$. Define the stochastic seminorm
\[\|Z\|^2_{M_t} = \Tr\left(Z^* \frac{d\langle M, M\rangle_t}{dt} Z\right)\]
where $\Tr$ denotes the trace and
\begin{equation} \label{eq:Mnorm}
 \frac{d\langle M, M\rangle_t}{dt} = \diag(A_tX_{t-}) - A_t\diag(X_{t-}) - \diag(X_{t-})A_t^*=:\Psi(A_t,X_t)
\end{equation}

 the matrix of derivatives of the quadratic covariation matrix of $M$. This seminorm has the property that
\[E\left[\int_{]0,t]} \|Z_u\|^2_{M_u} du\right] = E\left[\left(\int_{]0,t]} Z_u^*dM_u\right)^2\right]\]
for any predictable process $Z$ of appropriate dimension. We define the equivalence relation $\sim_{M}$ on the space of predictable processes by $Z\sim_{M} Z'$ if and only if $\|Z_t-Z'_t\|_{M_t} = 0$ $dt\times d\mathbb{P}$-a.s.
\end{definition}

\begin{remark}
A consequence of this choice of seminorm is that $Z+ c\bone \sim_M Z$ for any $Z$ and any predictable scalar process $c$. This is simply because $\sum_i e_i^* Ae_j =\sum_i A_{ij}= 0$ for all $j$, and so all row and column sums of $\Psi(A_t, X_t)$ are zero.
\end{remark}

\begin{theorem}
Every scalar square-integrable martingale $L$ can be written in the form
\[L_t = L_0 + \int_{]0,t]} Z_s^* dM_s\]
for some predictable process $Z$ taking values in $\bR^N$. The process $Z$ is unique up to equivalence $\sim_M$.
\end{theorem}
\begin{proof}
 See \cite{Cohen2008}.
\end{proof}

The key SDEs which we shall study are equations of the form
\begin{equation}\label{eq:dynamics}
dY_t = -f(\omega, t, Y_{t-}, Z_t) dt + Z_t^* dM_t
\end{equation}
where $f:\Omega\times\bR^+\times\bR\times\bR^N\to\bR$ is a progressively measurable function, $Y$ is an adapted process with $E[\sup_{t<T} Y_t^2]<\infty$ for all $T$ and $Z$ is predictable. A solution $Y$ to this equation is a special semimartingale, hence the canonical decomposition into a predictable part ($-\int f(\omega, t, Y_{t-}, Z_t) dt$) and a martingale part ($\int Z_t^* dM_t$) is unique.

We shall make the general assumption that $f(\omega,t,y,\cdot)$ is invariant with respect to equivalence $\sim_M$, that is,
\[\text{if }Z\sim_M Z'\text{ then }f(\omega,t,y,Z)= f(\omega,t,y,Z')\] up to indistinguishability. This assumption is important, as it ensures that $f$ only considers $Z$ in the same way as it affects the integral $\int Z^* dM$.

In terms of existence  and uniqueness of solutions to (\ref{eq:dynamics}) we shall focus on two key cases,
\begin{itemize}
 \item first, when $Y_0\in\mathbb{R}$ and $\{Z_t\}_{t\geq 0}$ a predictable process are given, and so (\ref{eq:dynamics}) is a forward SDE, and
 \item second, when $Y_T \in L^2(\F_T)$ is given for some $T>0$, and $Z_t$ is chosen to solve the backward SDE, that is, to ensure that $Y$ is adapted and $Z$ is predictable.
\end{itemize}
For the forward equation, the existence and uniqueness of solution process $Y$ is classical, under various assumptions on the driver $f$. For the backward equation, under the assumption of Lipschitz continuity of the driver $f$, a result on existence and uniqueness of the pair $(Y,Z)$ is given in \cite{Cohen2008} (where $Z$ is unique up to equivalence $\sim_M$, as defined in the following definition). In this paper, we shall not focus on determining conditions on $f$ such that existence and uniqueness of solutions holds, but shall always assume that sufficient conditions are placed on $f$ such that the equation of interest has a unique solution.

\begin{lemma}\label{lem:ZuniquesimM}
For any pair $(Y,Z)$ satisfying (\ref{eq:dynamics}), the pair $(Y, Z')$ also satisfies (\ref{eq:dynamics}) if and only if  $Z \sim_M Z'$. 
\end{lemma}
\begin{proof}
By the canonical decomposition of $Y$ into a finite variation and a martingale part, we see that $\int_{]0,t]} Z^* dM = \int_{]0,t]} Z'^*dM$ up to indistinguishability, and so $Z \sim_M Z'$. Conversely, if $Z \sim_M Z'$, then by our assumption on $f$, $(Y, Z')$ will also satisfy (\ref{eq:dynamics}), up to indistinguishability.
\end{proof}

\begin{lemma}\label{lem:jumptimesenough}
For any martingale $L$, let $\Delta L$ denote the jumps of $L$. If $Z$ is a predictable process such that $\Delta L_t = Z_t \Delta M_t$ up to indistinguishability, then $L_t = L_0 + \int_{]0,t]} Z_u^* dM_u$, and $Z$ is unique up to equivalence  $\sim_M$.
\end{lemma}
\begin{proof}
Let $L$ have a representation $dL_t = Z'_t dM_t$ for some predictable process $Z'$. Then as $\Delta L_t = Z_t \Delta M_t$ up to indistinguishability, we must have that $B=\{(\omega, t): \|Z_t- Z'_t\|_{M_t}\neq 0\}$ is a predictable set such that $\Delta M_t =0$ on $B$. One can verify that a predictable set with this property is the union of a $dt\times d\mathbb{P}$-null set and a set on which $A_tX_{t-} \equiv \mathbf{0}$. However, if $A_tX_{t-}=0$ then $d\langle M,M\rangle/dt = 0$, and so we see that $Z\sim_M Z'$. As the martingale representation is unique up to equivalence $\sim_M$, we have our result.
\end{proof}

\section{Markovian solutions to (B)SDEs}

\begin{definition}\label{def:Markovian}
We say a stochastic process $Y$ is \emph{Markovian} if, up to indistinguishability, $Y_t$ depends on $\omega$ only as a function of $X_t$, that is, it can be written as
\[Y_t \equiv u(t, X_t)\]
for some deterministic function $u$.
\end{definition}
We note that this is, in some ways, a misnomer, as a process satisfying Definition \ref{def:Markovian} need not be Markovian in the sense that $Y_t$ is conditionally independent of $Y_r$ given $Y_s$, for all $r<s<t$. Nevertheless, this terminology is standard in the theory of BSDEs, and describes adequately the property of interest.

Our key result on the structure of Markovian solutions to (\ref{eq:dynamics}) is the following.

\begin{theorem}\label{thm:markovstructure}
Suppose $Y$ has dynamics given by (\ref{eq:dynamics}). Then the following statements are equivalent:
\begin{enumerate}
\item  $Y$ is of the form $Y_t=u(t,X_t)$ for some function $u:\bR^+\times\bR^N\to\bR.$
\item  Both
\begin{enumerate}[(i)]
\item $f$ is Markovian, that is, it is of the form
\[f(\omega, t, Y_t, Z_t) = \tilde f(X_t,t,Y_t,Z_t) \quad \mathbb{P}\times dt-a.e.\]
for some function $\tilde f: \bR^N\times \bR^+\times\bR\times\bR^N\to\bR$ and
\item up to $\sim_M$ equivalence, $Z_t$ is a deterministic function of $t$ which satisfies $X_0^*Z_0 = Y_0$ and for all $i$
\[\frac{d (e_i^*Z_t)}{dt} = -\tilde f(e_i, t, e_i^*Z_t, Z_t)  - Z_t^* A_t e_i.\]
\end{enumerate}
\end{enumerate}
Furthermore, under either set of conditions, $e_i^*Z_t = u(t, e_i)$ up to equivalence $\sim_M$.
\end{theorem}

\begin{proof}
\emph{1 implies 2.} From (\ref{eq:dynamics}), $Y$ is continuous except at a jump of $X$. By boundedness of $A$, for any bounded interval $[a,b]$, with positive probability there will not be a jump of $X$ in $[a,b]$. Therefore, for each $i$, considering the non-null set $\{\omega: X_t=e_i\text{ for all }t\in[a,b]\}$ we see that $u(\cdot, e_i)$ must be continuous on every bounded interval $[a,b]$.

Consider a jump of $X$ which occurs at the stopping time $\tau$. Let $L_t := \int_0^t Z_u^*dM_u$, that is $L$ is the martingale part of $Y$. Then from (\ref{eq:dynamics})
\[ \Delta L_\tau = Y_\tau-Y_{\tau-} = Z_\tau^* \Delta M_t = Z_\tau^* (X_\tau-X_{\tau-}).\]
Define the process $Z'_t$ with components
\[e_i^*Z'_t = u(t, e_i).\]
Then we have $(Z'_\tau)^*X_\tau = Y_\tau$ and $(Z'_\tau)^*X_{\tau-} = Y_{\tau-}$, for every jump of $X$. Hence $Z'$ is a predictable (indeed, deterministic) process such that $\Delta L = Z'\Delta M$, and by Lemma \ref{lem:jumptimesenough}, we see $Z \sim_M Z'$.

Now note that, except on the thin set $\{\Delta X_t\neq 0\}$,
\[\frac{d}{dt}u(t, X_t) = \frac{dY_t}{dt} = -f(\omega, t, Y_t, Z_t) - Z_t^* A_t X_{t-}.\]
Considering these dynamics on the non-null set $\{\omega: X_t=e_i\text{ for all }t\in[a,b]\}$, we obtain
\begin{equation}\label{eq:fixdynamics}
 \frac{d}{dt}u(t, e_i) = -f(\omega, t, Y_t, Z_t) - Z_t^* A_t e_i.
\end{equation}
As $Z_t$ is deterministic, by rearrangement, we see that
\[f(\cdot, t, Y_t, Z_t) = - \frac{d}{dt}u(t, X_t) - Z_t^*A_t X_{t-}\]
 does not vary with $\omega$ on the sets where $X_t$ is constant. Therefore, we can write $\tilde f(X_t, t, Y_t, Z_t)$ as the common value taken by $f$ on these sets. Finally, we see that replacing $f$ by $\tilde f$  and using the fact that $e_i^*Z_t = u(t, e_i)$ in (\ref{eq:fixdynamics}) yields the desired dynamics for $Z$.

\emph{2 implies 1.} By uniqueness of solutions to (\ref{eq:dynamics}), as $Y_0= X_0^* Z_0$ and $Z$ has the prescribed dynamics, we know that $Y_t = X_0^*Z_t$ up to the first jump of $X$. At the first jump time $\tau$, we see that
\[Y_\tau-Y_{\tau-}=\Delta Y_\tau = Z_\tau^*(X_\tau - X_{\tau-}) = Z_\tau^*X_\tau - Y_{\tau-},\]
and so $Y_\tau = X_\tau^*Z_\tau$ almost surely. Repeating this argument and using induction on the (almost surely countable) sequence of jumps of $X$, we see that $Y_t = X_t^*Z_t$ up to indistinguishability. However $Z_t$ is deterministic, so we can define a deterministic function $u(t, e_i) = e_i^* Z_t$, and we see that $Y_t = u(t, X_t)$.
\end{proof}

The following corollary provides a frequently more convenient way to analyse such equations.
\begin{corollary}\label{cor:uODE}
Let $Y$ be as in Theorem \ref{thm:markovstructure}, with associated function $u$. Write $\bu_t$ for the column vector with elements $u(t, e_i)$. Write $\bf(t,\bu)$ for the column vector with elements $e_i^*\bf(t,\bu_t) := \tilde f(e_i, t, e_i^*\bu_t, \bu_t)$. Then $\bu$ satisfies the vector ordinary differential equation
\begin{equation}\label{eq:uODE}
d\bu_t = -(\bf(t,\bu_t) +A^* \bu_t) dt
\end{equation}
\end{corollary}
\begin{proof}
Simply note that $\bu = Z$ in the proof of the theorem, and so the dynamics are as given.
\end{proof}

\begin{corollary}
Let $Y$ be the solution to a BSDE with Markovian terminal condition $Y_T = \phi(X_T)$, for some deterministic function $\phi:\bR^N\to\bR$. Suppose $f$ is Markovian. Then for all $t<T$, $Y_t = u(t,X_t)$, where 
\begin{itemize}
\item $u(T,\cdot) = \phi(\cdot)$, 
\item the associated vector $\bu_t$ satisfies the ODE (\ref{eq:uODE}) and
\item the solution process $Z$ is given by $Z_t=\bu_t$. In particular, note that $Z$ is deterministic and continuous.
\end{itemize}
\end{corollary}
\begin{proof}
Simply define $\bu_t$ as the solution to the ODE (\ref{eq:uODE}), working backwards in time, with initial value $e_i^*\bu_T = \phi(e_i)$. Then the pair of processs $(Y_t, Z_t):=(X_t^*\bu_t, \bu_t)$ is a solution to the BSDE with dynamics (\ref{eq:dynamics}), and hence is unique (up to equivalence $\sim_M$ for $Z$). The remaining properties follow directly from the theorem.
\end{proof}

\subsection{Brownian BSDE and semilinear PDE}
It is worth comparing these results with those for BSDEs driven by Brownian motion. In the simplest classical case, suppose the filtration is generated by a scalar Brownian motion $W$, and consider a BSDE of the form
\[dY_t = -f(W_t, t, Y_t, Z_t) + Z_t dW_t; \quad Y_T=\phi(W_T)\]
where $\phi:\bR\to\bR$ is continuous. 

Then, one can show from the Feynman-Kac theorem (see, for example, \cite{El1997} or \cite{Peng1992}) that $Y_t = u(t, W_t)$, where $u$ is the viscosity solution to the semilinear PDE
\[\partial_t u(t, x) = -f(x, t, u(t,x), \partial_x u(t,x)) - \frac{1}{2} \partial^2_{xx}(u(t,x)), \quad u(T, x) = \phi(x).\]

Comparing this to our equation (\ref{eq:uODE}) for $\bu_t$, we see that first difference is that we have replaced the infinitesimal generator of the Brownian motion ($\frac{1}{2} \partial^2_{xx}$), by the infinitesimal generator of the Markov chain ($A^*_t$). 

The second difference is that, in the Brownian case, $f(x,t,\cdot,\cdot)$ depends only on the behaviour of $u$ in a neighbourhood of $x$. In the Markov chain case, as our state space does not have a nice topological structure, it would seem that $e_i^*\bf$ can depend on all the values of $\bu$, not only on those `close' to the $i$th coordinate of $\bu$.

However, when we examine $e_i^* \bf(t, \bu_t) = \tilde f(e_i, t, e_i^*\bu_t, \bu_t)$, we see that $e_i^*\bf$ depends only on the $i$th coordinate of $\bu$, and on those properties of $\bu$ which are invariant up to equivalence $\sim_M$. In particular, if a jump from state $e_i$ to state $e_j$ is not possible, then $e_i^*\bf$ cannot depend on the value of $e_j^*\bu_t$ (as changing this value will give a vector which is equivalent to $\bu_t$ up to equivalence $\sim_M$). Similarly, if a constant is added to every element of $\bu_t$, $e_i^*\bf$ will change only through the dependence on $e_i^*\bu_t$, rather than through any other element of the vector. In this sense, the `local' dependence is preserved by the equivalence relation.


An alternative way of thinking through the relationship between our result and those known in the Brownian case is through the following diagram, indicating how an equation of one type can be converted into another.
 \[
\begin{CD}
\text{Brownian (B)SDE}   @>\text{Feynman-Kac}>>   \text{Parabolic (semilinear) PDE}\\
@V\text{Space Discretisation}VV                                    @VV{\text{Finite Element Method}}V\\
\text{Markov chain (B)SDE}                 @>(*)>>      \text{Coupled ODE system}
\end{CD}
\]
Our result provides the link indicated by $(*)$. It is natural to think that, given appropriate choices of spatial discretisations and finite element methods, this diagram will commute.

\section{Calculating BSDE solutions}

As we have shown that there is a connection between BSDEs driven by Markov chains and systems of ODEs, it is natural to use this connection for the purposes of computation. As with classical BSDE, this connection can be exploited in both directions, depending on the problem at hand.

As mentioned before, various practical problems can be analysed using the framework of BSDE. Of particular interest are dynamic risk measures and nonlinear pricing systems, as described in \cite{Cohen2008b}. By connecting this theory with the theory of ODEs, we gain access to the large number of tools available for the numerical calculation of ODE solutions, the only concern being the dimensionality of the problem. We note that while we shall consider an example from finance, the same methods can be applied in other areas of stochastic control.

We give a practical example taken from Madan\footnote{Thanks to Dilip Madan for kindly providing us with access to the fitted data from this paper.}, Pistorius and Schoutens \cite{Madan2010}. In this paper a 1600-state Markov chain on a non-uniform spatial grid is created to match the behaviour of a stock price in discrete time, assuming that there exists an underlying Variance-Gamma local L\'evy process, based on the CGMY model of \cite{Carr2004}. The transition probabilities are fitted in discrete time from month to month, yielding a calibrated discrete time risk-neutral transition matrix.

From this matrix, we extract a continuous-time Markov chain approximation, using a carefully constructed approximation of the matrix logarithm. (An approximation is needed as the matrix in question is large, and possibly due simply to calibration error, does not exactly correspond to the skeleton matrix of a continuous time Markov chain. We hope to give the details of this approximation in future work.) For our purposes, the only relevant quantities are the grid used, (that is, the value of the underlying stock in each state), and the rate matrix of the Markov chain.

Let $S(X_t)$ denote the value of the stock in state $X_t$. We shall consider the risk-averse valuation of a contingent claim using our BSDE. To do this, we fix the terminal value as a function of the state $Y_T= \phi(S(X_T))$ for some function $\phi$. We then take the risk-neutral valuation $E[Y_T]$, and dynamically perturb this through the use of a BSDE with concave driver (if the driver were $f\equiv 0$, then we would simply obtain the risk-neutral price $E[Y_T]$). This yields a process $Y_t=u(t,X_t)$, which we interpret as the \emph{ask price} (that is, the amount an agent is willing to pay, at time $t$ in state $X_t$) of the terminal claim $Y_T$.

The value $Y_t$ can be thought of as containing both the risk-neutral price and a correction due to risk-aversion. Through the use of a BSDE, we ensure that this correction can be dynamically updated, and so our prices are consistent through time (they do not admit arbitrage, see \cite{Cohen2008b}).

If we think of $Y_t$ as the ask price, that is, the amount an agent is willing to pay to purchase $Y_T$, then it is natural to also ask how much he would be willing to sell $Y_T$ for, that is, the \emph{bid price}. This value corresponds to the negative of the solution of the BSDE with terminal value $-Y_T$ (as selling $Y_T$ is equivalent to purchasing $-Y_T$, and we change the sign of the final solution so that it represents an inward rather than outward cashflow).

To solve this equation, we then convert our BSDE with driver $f$ and terminal value $\phi(S(X_T))$ into a coupled system of ODEs, using Theorem \ref{thm:markovstructure}. It is then a simple exercise to use any standard ODE toolbox (for our examples we have used the {\tt ode45} IVP solver in Matlab) to solve the relevant ODE system. 

\begin{example}
 Consider the BSDE with driver
\[\tilde f(X_t, t, z)= \min_{r\in [\alpha^{-1}, \alpha]} \{r (z^* A_t X_t)\}.\]
This equation corresponds to uncertainty about the overall rate of jumping from the current state -- the parameter $\alpha\geq 1$ determines the scale of the uncertainty. The uncertainty is, however, only about the overall scale of the jump rate -- the relative rates of jumping into different states remain the same.

Note that this driver is concave, and satisfies the requirements of the comparison theorem in \cite{Cohen2008b}. Hence the solutions to this BSDE give a `concave nonlinear expectation' $\E(Q|\F_t):=Y_t$ in the terminology of \cite{Cohen2008b}. The driver is also positively homogenous (that is, $f(X_t, t, \lambda z) = \lambda f(X_t, t, z)$ for all $\lambda>0$) and so the nonlinear expectation is positively homogenous, that is, it does not depend on the units of measurement.

For our numerical example, we use a timeframe of one month, set $\alpha=1.1$, and calculate the value of a simple Butterfly spread
\[\phi(s) = \begin{cases}
             s &s\in[15,20[\\
	     25-s& s\in[20, 25[\\
	     0& \text{otherwise.}
            \end{cases}
\]

\begin{figure}[htcb]
\begin{center}
\includegraphics[width=10cm]{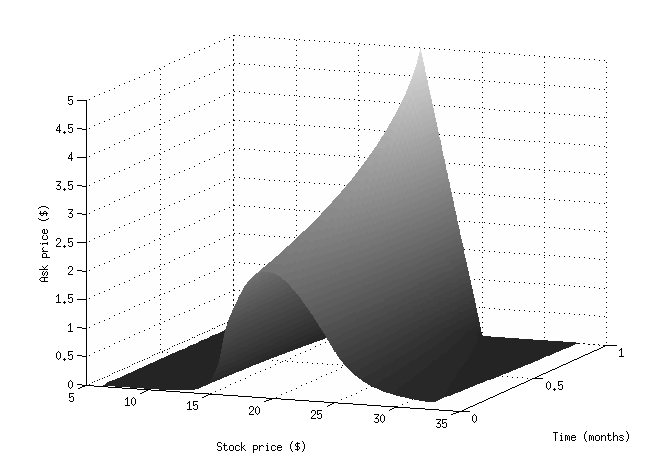}
\caption{Ask price surface for a butterfly spread under rate uncertainty.}\label{fig:1}
\end{center}

\begin{center}
\includegraphics[width=8cm]{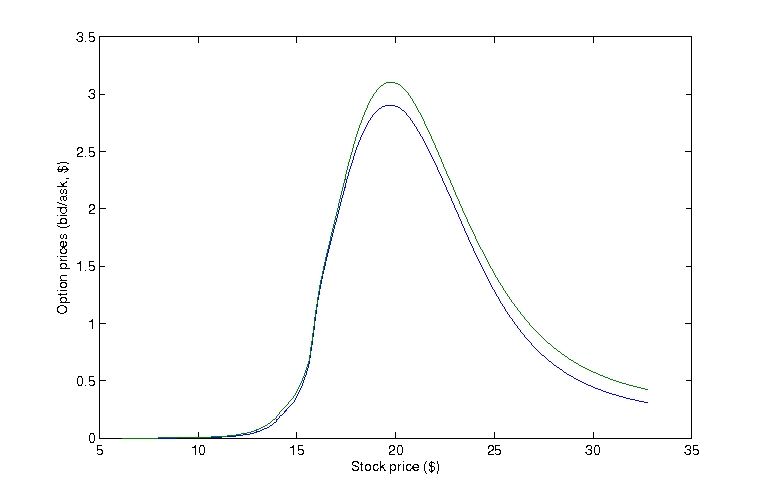}
\caption{Time $t=0$ bid (high) and ask (low) prices for a butterfly spread under rate uncertainty.}\label{fig:2}
\end{center}
\end{figure}

The results of solving the resultant system of ODEs can be seen in Figures \ref{fig:1} and \ref{fig:2}. In Figure \ref{fig:1} we plot the surface generated by plotting the solution of the $i$th term of the ODE against the stock price $S(e_i)$. The assymetry that can be seen in Figure \ref{fig:1} is due to assymetry in the rate matrix of the markov chain -- the volatility of the stock depends on its current level. The prices look qualitatively similar to what one would obtain using the classic expectation (which, in fact would lie between the bid and ask curves in Figure \ref{fig:2}). As one would hope, the bid price lies above the ask price, however we can see that the difference (the bid-ask spread) is not constant, and is higher when the stock price is larger.

An aspect of these prices which is not immediately apparent from the figures is that they are dynamically consistent, that is, if prices evolve in this manner, then one cannot make an arbitrage profit. This follows from the fact that these prices satisfy a BSDE, see \cite{Cohen2008b}.
\end{example}

\begin{example}
 In \cite{Madan2010}, various prices are determined using the estimated discrete time model and a nonlinear pricing rule. The technique used is to apply a concave distortion the the cumulative distribution function of the values one-step ahead. That is, the value at time $0$ of a payoff at time $1$ with cdf $F$ is given by
\[\int_\mathbb{R} x d(1-(1-F(x)^{\frac{1}{1+\gamma}})^{(1+\gamma)})\]
for some `stress level' $\gamma>0$. This is called the \emph{minmaxvar} distortion.

In the same vein, we now consider the use of the driver
\[\tilde f(X_t, t,z)= z^*(\tilde A_t^z-A_t)X_t;\]
where $\tilde A_t^z$ arises from the continuous time analogue of minmaxvar, where we distort the relative rates of jumps to each of the non-current states. This is defined by the following algorithm:

\medskip
\textbf{Minmaxvar rate matrix distortion}
\begin{enumerate}
  \item Sort the components of $z_i$ to give an increasing sequence $z_{\pi(i)}$, where $\pi$ is a permutation of $\{1,...,N\}$.
 \item Define the cumulative sum of the corresponding sorted rates (excluding the current state), 
$G(i)= \sum_{j=1}^i (e_{\pi(j)}^* A_t X_t)^+.$
\item Apply a concave distortion to the scaled cumulative sum
\[\psi(G(i)) = \Bigg(1-\Bigg(1-\Bigg(\frac{G(i)-G(1)}{G(N)-G(1)}\Bigg)^{\frac{1}{1+\gamma}}\Bigg)^{1+\gamma}\Bigg)(G(N)-G(1)) +G(1).\]
\item Define the individual distorted rates $q_i = \psi(G(i))-\psi(G(i-1))$, with initial term $q_1=G(1)$;
\item Unsort these rates to define a distorted rate matrix $\tilde A^z_t$
\[e_i^*\tilde A_t^z X_t = q_{\pi^{-1}(i)}, \qquad \{e_i\neq X_t\}\]
with the diagonal then chosen to give row sums of zero.
\end{enumerate}
\medskip

We apply this pricing mechanism to a digital option with a knockout barrier, that is, a payoff $\phi$ where
\[\phi=\begin{cases} 0 & \text{if }S(X_t)\text{ is ever above }25\\
        1& \text{if }S(X_T)> 15 \text{ and }S(X_t)<25\text{ for all }t\\
	0& \text{otherwise}
       \end{cases}
\]
This type of simple barrier option is straightforward to calculate, as we simply ensure that our solution satisfies the additional boundary condition $u(t, X_t)=0$ for all states where $S(X_t)\geq 25$. Our solution is then the value of the option, conditional on the barrier not having been hit. The results of this, using a stress level $\gamma=0.1$, can be seen in Figures \ref{fig:3} and \ref{fig:4}.

\begin{figure}[htcb]
\begin{center}
\includegraphics[width=10cm]{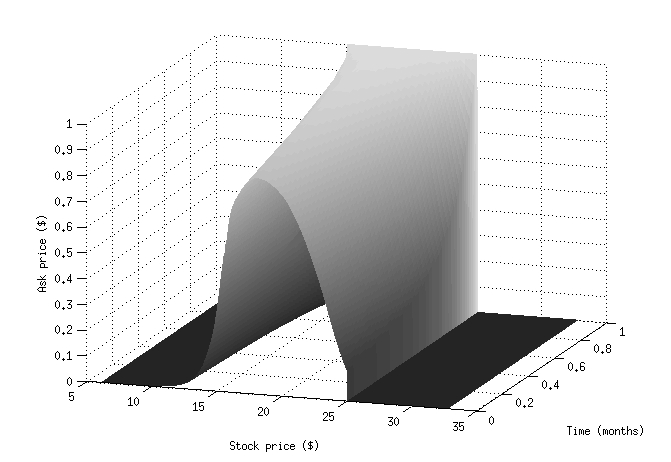}
\caption{Ask price surface for a digital knockout option under minmaxvar, $\gamma=0.1$.}\label{fig:3}
\end{center}

\begin{center}
\includegraphics[width=8cm]{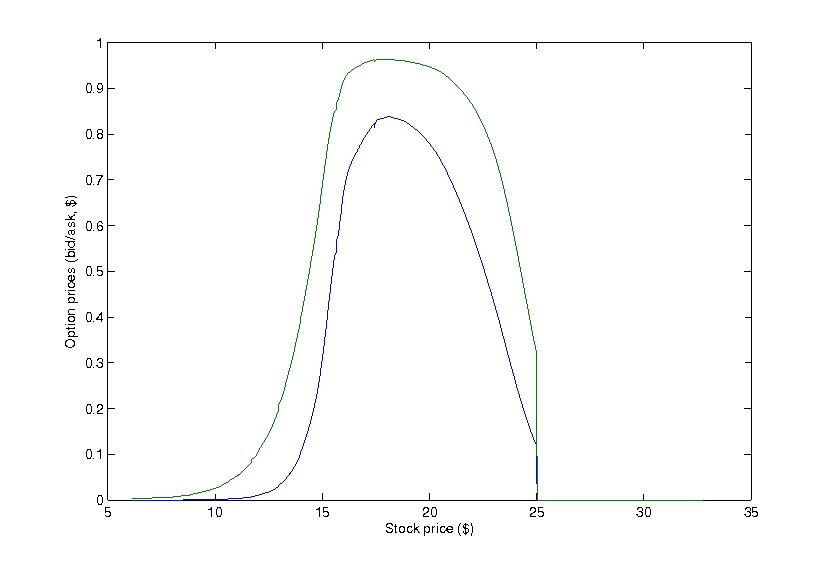}
\caption{Time $t=0$ bid (high) and ask (low) prices for a digital knockout option under minmaxvar, $\gamma=0.1$.}\label{fig:4}
\end{center}
\end{figure}

Again we can see the effects of risk aversion in the differences of the bid and ask prices in Figure \ref{fig:4}. The effects of the knockout barrier are also clear, as it causes a sharp change in the prices at the boundary.
\end{example}

\section{Calculating ODE solutions}

It is well known that a potential application of the theory of Brownian BSDEs is to provide stochastic methods for large PDE systems. In our situation, we have seen that the natural relation is not between a BSDE and a PDE, but a BSDE and a system of coupled ODE. We therefore can consider adaptations of the stochastic methods for solving BSDE as candidates for novel schemes for solving large systems of coupled nonlinear ODE. As numerical algorithms for ODEs are generally very good, we do not expect that this method will be of use except in some extreme cases. For this reason, we simply outline the algorithm.

To implement such a method, consider the following general setting. Suppose we have a system of ODEs of the form
\begin{equation}\label{eq:forODE}
 d\mathbf{v}_t = g(t, \mathbf{v}_t) dt; \quad \mathbf{v}_0= \phi,
\end{equation}
and the object of particular interest is the value of $e_k^* \mathbf{v}_T$, that is, the value of the $k$th component of $\mathbf{v}$ at some future time $T$. (This method can, of course, be modified to give all components, however is particularly well suited to when our interest is in a single component.)

We split the ODE (\ref{eq:forODE}) into the form
\begin{equation}\label{eq:forODEsplit}
 d\mathbf{v}_t = (\mathbf{f}(t, \mathbf{v}_t) + A_t^*\mathbf{v}_t) dt
\end{equation}
where $A_t$ is a rate matrix (that is, a matrix with nonnegative entries off the main diagonal and all row sums equal to zero) and $\mathbf{f}$ is a function with the property that there exists $c$ such that
\[|e_i^*(f(t, \mathbf{v}_t) - f(t, \mathbf{v}'_t))|^2 \leq c (\mathbf{v}_t- \mathbf{v}'_t)^* \Psi(A_t, e_i) (\mathbf{v}_t- \mathbf{v}'_t)\]
for all $i$, where $\Psi(\cdot, \cdot)$ is as in (\ref{eq:Mnorm}). Practically, this means that the $i$th component of $f$ can only depend on the $j$th component of $\mathbf{v}$ when $(A_t)_{ij} >0$.

We then reverse time, defining $\mathbf{u}_t = \mathbf{v}_{T-t}$, so that (\ref{eq:forODEsplit}) becomes precisely the ODE we obtain from our BSDE (\ref{eq:uODE}). Therefore, we have converted our problem into determining the initial value of a BSDE with terminal value $X_t^*\phi$, when $X$ is a Markov chain with rate matrix $A_t$. Furthermore, as our interest is in the value of the $k$th component of $v$, we are interested in the initial value of our BSDE when $X_0=e_k$.

We now outline the natural modification of the algorithm of Bouchard and Touzi \cite{Bouchard2004} for the calculation of solutions to these BSDEs. We shall not give a proof of the convergence of this algorithm, however it is natural to believe that the conditions for convergence from the Brownian setting will carry over to the setting of Markov Chains.

\medskip
\textbf{Monte-Carlo Algorithm for Markov Chain BSDE}

First fix a discretisation level $\Delta t=T/M$.
\begin{enumerate}
 \item (Forward step) Simulate a large number paths of the Markov chain beginning in state $X_0=e_k$, on the discretised grid $\{k\Delta t\}_{k\leq M}$. Write $X_{t}^n$ for the value of the $n$th simulation of the Markov chain at time $t$. The method of simulating this forward step can be chosen for the sake of computational convenience. The cost of doing this will depend on the size of the components of the matrix $A_t$.
 \item (Backward step) Iterate backwards from $T$ to $0$ in steps of size $\Delta t$, at each step, we seek to create an estimate of the function $u_t= u(t, \cdot):\{1,2,...,N\}\to\mathbb{R}$ based on the function $u_{t+1}$ (where $N$ is the number of ODEs).
\begin{enumerate}
\item Use the function $u_{t+1}$ to calculate the values of $f^n:=(X_{t+1}^n)^*f(t, u_{t+1})$ (recall that as $X^n_{t+1}$ takes values from the standard basis vectors of $\mathbb{R}^N$, this simply selects out some terms of the function $f$.
\item Calculate
\[\hat u_{t+1}^n= u_{t+1}(X_{t+1}^n) - f^n\Delta t.\]
We wish to approximate
\[u_t(x)\approx E[\hat u_{t+1}(X_{t+1})|X_t=x]\]
This can be done using a Longstaff-Schwarz technique (see \cite{Longstaff2001}), using an appropriate basis $\{\psi_m\}$ for functions on $\{1,2,...,N\}$. To do this, we choose coefficients $c_m$ to minimise
\[\sum_n \left(\sum_m(c_m\phi_m(X_{t+1}^n))-\hat u_{t+1}(X_{t+1}^n)\right)^2\]
This gives the least-squares approximation
\[u_t(x) = \sum_m(c_m\phi_m(x)).\]
\item Repeat until an estimate for $u_0$ is obtained.
\end{enumerate}
We note that the accuracy and cost of the backward step primarily depends on the size and computational cost of evaluating $f$.
\end{enumerate}
This algorithm is particularly well suited for calculating the values $u_0(e_k)=e_k^*\mathbf{v}_T$, as the simulated paths can be chosen to start in state $e_k$. Hence the algorithm will not attempt to accurately calculate values of $u_t(e_{k'})=e_{k'}\mathbf{v}_{T-t}$ for combinations of $t$ and $k'$ which have little impact on $e_k^*\mathbf{v}_T$. We also note that there is a trade off between the costs of the forward and backward processes, depending on the decomposition from (\ref{eq:forODE}) to (\ref{eq:forODEsplit}). Typically, it seems that the forward process is relatively cheap to simulate accurately, so it is preferable to try and minimise the size of $f$.

\section{Conclusion}

We have considered solutions to Backward Stochastic Differential Equations where noise is generated by a Markov chain. We have seen that solutions to these equations are markovian, that is, they can be expressed as a deterministic function of the current state, if and only if they come from the solution of a certain coupled ODE system.

Consequently, calculating the Markovian solutions to these BSDEs is a simple matter of evaluating an initial value problem, for which many good numerical methods exist. This has applications to calculating risk-averse prices (for example, bid-ask prices) for market models driven by Markov chains. We have also seen how this suggests new stochastic methods for the solution of large ODE systems.

\bibliographystyle{plain}
\bibliography{../../RiskPapers/General}
\end{document}